\documentclass[11pt]{amsart}

\usepackage{amsthm}
\usepackage{graphicx}
\usepackage{enumerate}
\usepackage{amsmath,amssymb,amscd,amsfonts}
\usepackage{mathtools}
\usepackage{color}
\usepackage{lineno}
\allowdisplaybreaks[4]

\newtheorem{theorem}{Theorem}[section]
\newtheorem{lemma}[theorem]{Lemma}

\theoremstyle{definition}

\theoremstyle{remark}
\newtheorem{remark}[theorem]{Remark}

\numberwithin{equation}{section}

\newcommand{\R}{\mathbb{R}}

\newcommand{\cF}{\mathcal{F}}

\newcommand{\cL}{\mathcal{L}}
\newcommand{\cN}{\mathcal{N}}

\newcommand{\al}{\alpha}
\newcommand{\be}{\beta}

\begin{document}


\title[A generalization of the Takagi function]{A generalization of the Takagi function for beta-expansions}

\author{Shintaro Suzuki}
\address{Department of Mathematics, 
Tokyo Gakugei University, 
4-1-1 Nukuikitamachi, 
Koganei-shi, 
Tokyo 184-8501,
Japan}
\email{shin05@u-gakugei.ac.jp}


\begin{abstract}
We consider a generalized Takagi function for beta-expansions with the base $1<\be\leq2$, motivated by multifractal analysis for digit frequency sets of beta-expansions \cite{Su}. We show that it is pointwise  $\al$-H\"older continuous for any $\al\in(0,1)$ but not pointwise Lipschitz continuous on the unit interval except a Lebesgue null set. Our proof relies on a formula for the generalized Takagi function reflecting its oscillations of the sum of digits and some basic limit theorems for the corresponding beta-map.   
\end{abstract}

\maketitle
\section{Introduction}

The Takagi function $\mathcal{T}:[0,1]\to[0,1]$ is a simple example of a continuous but nowhere differentiable function defined by  
\[\mathcal{T}(x)=\sum_{n=1}^\infty\frac{T^n(x)}{2^n}\] for $x\in[0,1]$, where $T(x)$ is the tent map $T(x)=1-|1-2x|$ (see e.g., \cite{Al-Ka, La, Ta}). Its fractal properties have been well-studied, for example, non-Lipschitz continuity and $\alpha$-H\"older continuity for $0<\alpha<1$ \cite{Br-Ko, Sh-Sa}, a characterization of points where it has the infinite derivative \cite{Al-Ka0}, some properties of its level sets \cite{Al, La-Mo} 
and the limit theorems for its oscillations \cite{Ko}. For showing these properties, 
some different expressions of the Takagi function as described in \cite{Al-Ka, La} can be applicable, each of which has its own characteristics. In this paper, we focus on one of such expressions given by
\[\mathcal T(x)=\frac{d_1(x)}{2}+\sum_{n=2}^\infty \frac{d_n(x)}{2^n}\Bigl(n-2\sum_{i=1}^{n-1}d_i(x)\Bigr)\ \]
for $x\in[0,1]$, where $d_k(x)\in\{0,1\}$ denotes the $k$-th coefficient of the dyadic expansion of $x\in[0,1]$ with $x=\sum_{n=1}^\infty d_n(x)/2^n$ 
 (see e.g., \cite[Section 4]{Al-Ka}). We consider its generalization $G_\be$ for beta-expansions with the base $1<\be\leq 2$ (see Section 2 for the definition), which coincides with $\mathcal T$ when $\be=2$ (cf. \cite{Ma-Wo, Ok}). The function $G_\be$ naturally appears in the context of multifractal analysis of digit frequency functions for beta-expansions (see \cite[Section 6]{Su}).
 The author in \cite[Section 6]{Su} showed that $G_\be$ is continuous but nowhere differentiable on the unit interval, from which we can regard this function as a natural generalization of the Takagi function. As a consecutive work, in this paper, we show that the function $G_\be$ is pointwise $\alpha$-H\"older continuous for any $\alpha\in(0,1)$ on the unit interval except a Lebesgue null set (Theorem \ref{3-3}). A key ingredient of the proof is a certain formula for $G_\be$ reflecting the oscillations of the sum of digits (Lemma \ref{1}). Using this formula and the first Borell-Cantelli lemma, we show that the H\"older coefficient of $G_\be(x)$ at almost all $x\in[0,1]$ (in the sense of the Lebesgue measure) is bounded by a positive constant related to the orbit of $x$ by the corresponding beta-map. In addition, we also show that the function $G_\be$ is not pointwise Lipschitz continuous on the unit interval except a Lebesgue null set (Theorem \ref{4-1}). To show this, for almost all $x\in[0,1]$, we give a sequence of real numbers $(x_N)_{N=1}^\infty$ which converges to $x$ satisfying $|G_\be(x_{N+1})-G(x_N)|/|x_{N+1}-x_N|\to\infty$ as $N\to\infty$ via the central limit theorem for a beta-map, which was established for more general piecewise $C^2$ expanding maps on the interval (see e.g., \cite{Bo-Go, Ho-Ke}). 

This paper is organized as follows. In Section 2, we summarize notions related to beta-expansions and define the generalized Takagi function $G_\be$. In Section 3, we show that it is pointwise $\alpha$-H\"older continuous for any $\alpha\in(0,1)$ on $[0,1]$ except a Lebesgue null set using the first Borel-Cantelli lemma. 
In Section 4, we show that the function $G_\be$ is not pointwise Lipschitz continuous on $[0,1]$ except a Lebesgue null set via the central limit theorem for the dynamics of a beta-map.

\section{Preliminaries}
In this section, we recall some necessary notions for beta-expansions following  \cite{Bl, It-Ta,  Pa, Re}. 
For $1<\be\leq2$, the beta-map is defined by
\[\tau_\be(x)=
\be x-[\be x] 
\]
for $[0,1]$, where $[y]$ denotes the integer part of $y\geq0$. 
This map is known as a simple example of piecewise linear expanding maps of the interval and has a unique unique invariant probability measure $m_\be$ absolutely continuous with respect to the Lebesgue measure $\cL$ (see \cite{Re}). In addition, its Radon-Nikodym derivative has the form 
\begin{equation}\label{Radon}
    \frac{d m_\be}{d\cL}=\frac{1}{F_\be}\sum_{n=0}^\infty \frac{\chi_{[0,\tau_\beta^n(1)]}}{\be^n},
\end{equation}
where $\chi_A$ denotes the indicator function of $A$ and $F_\be$ denotes the normalizing constant (see e.g., \cite{Ge, Pa}). It is known that the measurable dynamics $(\tau_\beta, m_\be)$ is ergodic (see e.g., \cite{Pa}). In fact, it is exact (see \cite{Bo, Sm}). 
The map $\tau_\beta$ gives the greedy expansion of $x\in[0,1]$ as follows. The definition of $\tau_\beta$ gives
\[x=\frac{[\be x]}{\be}+\frac{\tau_\beta(x)}{\be}\]
for $x\in[0,1]$, which yields
\[\tau_\beta^{n}(x)=\frac{[\be \tau_\beta^n(x)]}{\be}+\frac{\tau_\beta^{n+1}(x)}{\be}\]
for $n\geq0$. 
Using the above equations inductively, we have
\[x=\sum_{n=1}^N\frac{[\be\tau_\beta^{n-1}(x)]}{\be^n}+\frac{\tau_\beta^N(x)}{\be^N}\]
for $N\geq1$. Set $g_n(\be,x)=[\be\tau_\beta^{n-1}(x)]$ for $n\geq1$. Taking $N\to\infty$ in the right side of the above equation yields the greedy expansion of $x$:  
\[x=\sum_{n=1}^\infty\frac{g_n(x)}{\be^n},\]
which is equal to the binary expansion if $\be=2$. 
A number $x\in[0,1]$ is said to be simple if there is a positive integer $n_0\geq1$ such that $\tau_\beta^{n_0-1}(x)=1/\be$. In this case, we have $g_n(x)=[\be\tau_{\beta}^{n-1}(x)]=0$ for all $n\geq n_0+1$. 

Related to multifractal analysis for digit frequency  sets of beta-expansions, the author in \cite[Section 6] {Su} defined a generalized Takagi function $G_\be$ for $1<\be\leq2$ by
\begin{equation}
G_\be(x)=\frac{g_1(x)}{\be}+\sum_{n=2}^\infty\frac{g_n(x)}{\be^{n}}\Biggl(n-\frac{\sum_{i=1}^{n-1}g_i(x)}{M_\be}\Biggr)
\end{equation}
for $x\in[0,1]$, where 
\[M_\be=\int_0^1 g_1(x) d m_\be=m_\be([1/\be,1]).\]
In \cite{Su}, it was shown that $G_2=\mathcal T$ and the function $G_\be$ is continuous but nowhere differentiable on the unit interval for any $1<\be\leq2$, which yields $G_\be$ can be regarded as a generalization of the Takagi function $\mathcal T$. 

The following formula for $G_\be$ plays an important role in the proof of the main results of the paper.
\begin{lemma}\label{1} We have 
\begin{equation}
    G_\be(x)=\frac{1}{M_\be}\Biggl(x+\sum_{n=1}^\infty\frac{g_n(x)}{\be^n}\Bigl(M_\be n-\sum_{i=1}^ng_i(x)\Bigr)\Biggr)
\end{equation}
    for $x\in[0,1]$.
\end{lemma}

\begin{remark}
    In the case of $\be=2$, this formula is given in \cite[Section 4]{Al-Ka}. 
\end{remark}

\begin{proof}
Since 
\[g_n(x)\Biggl(n-\frac{\sum_{i=1}^{n-1}g_i(x)}{M_\be}\Biggr)=g_n(x)\Biggl(n-\frac{\sum_{i=1}^{n}g_i(x)}{M_\be}+\frac{1}{M_\be}\Biggr)\]
for $x\in[0,1]$ and $n\geq2$, we have
\begin{align*}
    G(x)&=\frac{g_1(x)}{\be}+\sum_{n=2}^\infty\frac{g_n(x)}{\be^{n}}\Biggl(n-\frac{\sum_{i=1}^{n-1}g_i(x)}{M_\be}\Biggr) \\
    &=\frac{g_1(x)}{\be}+\sum_{n=2}^\infty\frac{g_n(x)}{\be^{n}}\Biggl(n-\frac{\sum_{i=1}^{n}g_i(x)}{M_\be}+\frac{1}{M_\be}\Biggr) \\
    &=\sum_{n=1}^\infty\frac{g_n(x)}{\be^{n}}    \Biggl(n-\frac{\sum_{i=1}^{n}g_i(x)}{M_\be}+\frac{1}{M_\be}\Biggr) \\
    &=\frac{1}{M_\be}\sum_{n=1}^\infty\frac{g_n(x)}{\be^{n}}    \Biggl(M_\be n-\sum_{i=1}^{n}g_i(x)+1\Biggr)  \\
    &=\frac{1}{M_\be}\Biggl(x+\sum_{n=1}^\infty\frac{g_n(x)}{\be^n}\Bigl(M_\be n-\sum_{i=1}^ng_i(x)\Bigr)\Biggr)
    \end{align*}
for $x\in[0,1]$, as desired. 
\end{proof}

\section{H\"older continuity}
In this section, we show that the function $G_\be$ is pointwise $\alpha$-H\"older continuous for any $\al\in(0,1)$ on $[0,1]$ except some Lebesgue null set. For $x,y\in[0,1]$ with $x\neq y$, we define the orbit separation time $N(x,y)$ by 
\[N(x,y)=\min\{i\geq1; g_i(x)\neq g_i(y)\}.\]

We need the following two lemmas. 
\begin{lemma}\label{2}
    Let $x\in(0,1)$ be non-simple. 

    (1) For $y\in(0,1)$ with $x>y$, we have
    \[x-y\geq \frac{\tau_\beta^N(x)}{\be^N},\]
    where $N=N(x,y)$.

    (2) For $z\in(0,1)$ with $x<z$, we have
    \[z-x\geq \frac{1-\tau_\beta^{N}(x)}{\be^{N}},\]
    where $N=N(x,z)$.\end{lemma}

\begin{proof}
    (1) By the definition of the separation time $N=N(x,y)$, we know that $g_i(x)=g_i(y)$ for $1\leq i<N$ and $1=g_N(x)>g_N(y)=0$. Hence
    \begin{align*}
        x-y&=\sum_{n=1}^\infty\frac{g_n(x)}{\be^n}-\sum_{n=1}^\infty\frac{g_n(y)}{\be^n}    
        =\sum_{n=N}^\infty\frac{g_n(x)}{\be^n}-\sum_{n=N}^\infty\frac{g_n(y)}{\be^n}     \\
        &=\frac{1}{\be^N}+\sum_{n=N+1}^\infty\frac{g_n(x)}{\be^n}-\sum_{n=N+1}^\infty\frac{g_n(y)}{\be^n}     \\
        &=\frac{1}{\be^N}+\frac{\tau_\beta^N(x)}{\be^N}-\frac{\tau_\beta^N(y)}{\be^N}        \\
        &\geq\frac{\tau_\beta^N(x)}{\be^N},
        \end{align*}
        which ends the proof. 
        
    (2) By the definition of the separation time $N=N(x,z)$, we know that $g_i(x)=g_i(z)$ for $1\leq i<N$ and $0=g_N(x)<g_N(z)=1$. Hence
    \begin{align*}
        z-x&=\sum_{n=1}^\infty\frac{g_n(z)}{\be^n}-\sum_{n=1}^\infty\frac{g_n(x)}{\be^n} 
        =\sum_{n=N}^\infty\frac{g_n(z)}{\be^n}-\sum_{n=N}^\infty\frac{g_n(x)}{\be^n}     \\
        &=\frac{1}{\be^N}+\sum_{n=N+1}^\infty\frac{g_n(z)}{\be^n}-\sum_{n=N+1}^\infty\frac{g_n(x)}{\be^n}     \\
        &=\frac{1}{\be^N}+\frac{\tau_\beta^N(z)}{\be^N}-\frac{\tau_\beta^N(x)}{\be^N}        \\
        &\geq\frac{1-\tau_\beta^N(x)}{\be^N},
        \end{align*}
        which ends the proof. \end{proof}

\begin{lemma}\label{3}
    (1)  We have 
    \[\frac{1}{n}\log_{\be} \tau_\beta^n(x)\to 0\]
    as $n\to\infty$ for $\cL$-a.e. $x\in[0,1]$.

    (2) We have 
    \[\frac{1}{n}\log_{\be} (1-\tau_\beta^n(x))\to 0\]
    as $n\to\infty$ for $\cL$-a.e. $x\in[0,1]$.
    \end{lemma}

\begin{proof}
(1) Note that the set of all simple numbers $x\in[0,1]$ such that $\tau_\be^n(x)=1/\be$ for some $n\geq0$ is countable, which yields the set of all non-simple numbers is a Lebesgue full set on $[0,1]$. For a non-simple number $x\in(0,1]$, we have $-\infty<\log_\be\tau_\be^n(x)\leq0$ for any $n\geq1$, which gives 
\begin{equation}
    -\infty<\frac{1}{n}\log_\be\tau_\be^n(x)\leq0\end{equation}
for $\cL$-a.e. $x\in[0,1]$ and any $n\geq0$. 

Set  
\[A_n=\bigl\{x\in[0,1]; \tau_\be^n(x)\in[0, 1/n^2]\bigr\}=\tau_\be^{-n}([0,1/n^2])\]
for $n\geq1$. Note that for any subinterval $[a,b]\subset[0,1]$ with $0\leq a<b\leq1$ we have
\[m_\be([a,b])\leq \frac{1}{F_\be}\frac{b-a}{1-1/\be}\]
by the form of the Radon-Nikodym derivative (\ref{Radon}) of $m_\be$, where $F_\be$ is the normalizing constant. Hence
\begin{align*}
    \sum_{n=1}^\infty m_\be(A_n)
    &= \sum_{n=1}^\infty m_\be(\tau_\be^{-n}([0,1/n^2]) \\
    &=\sum_{n=1}^\infty m_\be([0,1/n^2]) \\
    &\leq \frac{1}{F_\be}\frac{1}{1-1/\be}\sum_{n=1}^\infty \frac{1}{n^2}<\infty,
\end{align*}
which yields $\displaystyle{m_\be\Biggl(\bigcap_{n=1}^\infty\bigcup_{k=n}^\infty A_k\Biggr)}=0$ by the first Borel-Cantelli lemma. Since the measure $m_\be$ is equivalent to the Lebesgue measure $\cL$, we have $\displaystyle{\cL\Biggl(\bigcap_{n=1}^\infty\bigcup_{k=n}^\infty A_k\Biggr)}=0$. That is, for $\cL$-a.e. $x\in[0,1]$ there is $n_0\geq1$ such that $\tau_\be^n(x)\geq 1/n^2$ for any $n\geq n_0$, which shows 
\[\frac{1}{n}\log_\be\tau_\be^n(x)\geq \frac{-2\log_\be n}{n}\to0\]
as $n\to\infty$. This gives the conclusion.  

(2) For $x\in[0,1]$ and $n\geq1$, we know that \[\frac{1}{n}\log_\be(1-\tau_\beta^n(x))\leq0.\]

Set \[B_n=\bigl\{x\in[0,1]; \tau_\be^n(x)\in[1-1/n^2,1]\bigr\}=\tau_\be^{-n}([1-1/n^2,1])\]
for $n\geq1$. As in the same way of the proof of (1), we have $\sum_{n=1}^\infty m_\be(B_n)<\infty$. This shows $\displaystyle{m_\be\Biggl(\bigcap_{n=1}^\infty\bigcup_{k=n}^\infty B_k\Biggr)}=0$ by the first Borel-Cantelli lemma, which yields  $\displaystyle{\cL\Biggl(\bigcap_{n=1}^\infty\bigcup_{k=n}^\infty B_k\Biggr)}=0$ by the fact that $m_\be$ is equivalent to $\cL$. That is, for $\cL$-a.e. $x\in[0,1]$ there is $n_0\geq1$ such that $1-\tau_\be^n(x)\geq 1/n^2$ for any $n\geq n_0$, which shows 
\[\frac{1}{n}\log_\be(1-\tau_\beta^n(x))\geq\frac{-2\log_\be n}{n}\to0\]
as $n\to\infty$.  This gives the conclusion. 
\end{proof}

The main theorem in this section is the following:

\begin{theorem}\label{3-3}  
There is a Lebesgue full set $\cF$ in $[0,1]$ such that the function $G_\be$ is pointwise $\al$-H\"older continuous for any $\al\in(0,1)$ on $\cF$. 
\end{theorem}

\begin{proof}
    First, we show that there is a Lebesgue null set $\cN\subset[0,1]$ such that $G_\be$ is $\al$-H\"older continuous from the left for any $\al\in(0,1)$ on $[0,1]\setminus \cN$. 
    By Lemma \ref{3}(1), we can take a Lebesgue null set $\cN$ such that  
    \[\frac{1}{n}\log_{\be} \tau_\beta^n(w)\to 0\]
    as $n\to\infty$ for any $w\in[0,1]\setminus \cN$. Let $x\in(0,1)\setminus \cN$ and take  $y\in(0,1)$ so that $x>y$. We write $N=N(x,y)$ for short. By Lemmas \ref{1} and \ref{2}(1), we obtain 
    \begin{equation}\label{above1}
    \begin{split}
        \frac{|G(x_{})-G(y_{})|}{(x-y)^\al} 
        &\leq \frac{1}{(x-y)^\al M_\be}\Biggl|(x-y)+\sum_{n=N}^\infty\frac{ng_n(x)}{\be^n}\Bigl(M_\be-\frac{1}{n}\sum_{i=1}^ng_i(x)\Bigr)\\
        &-\sum_{n=N}^\infty\frac{ng_n(y)}{\be^n}\Bigl(M_\be-\frac{1}{n}\sum_{i=1}^ng_i(y)\Bigr)\Biggr|\\
        &\leq \frac{(x-y)^{1-\al}}{M_\be}+\frac{\be^{\al N}}{M_\be(\tau_\beta^N(x))^\al}\Biggl|\sum_{n=N}^\infty\frac{ng_n(x)}{\be^n}\Bigl(M_\be-\frac{1}{n}\sum_{i=1}^ng_i(x)\Bigr) \\
        &-\sum_{n=N}^\infty\frac{ng_n(y)}{\be^n}\Bigl(M_\be-\frac{1}{n}\sum_{i=1}^ng_i(y)\Bigr)\Biggr| \\
        & \leq \frac{1}{M_\be}+\frac{\be^{\al N}}{M_\be (\tau_\be^N(x))^\al}\sum_{n=N}^\infty \frac{4n}{\be^n} \\
        &\leq \frac{1}{M_\be}+\frac{K_\be}{M_\be\be^{(1-\al)N}}\frac{1}{(\tau_\beta^N(x))^\al}, 
    \end{split}
    \end{equation}
    where $K_\be$ is a positive constant independent of $N$.
    Since 
    \begin{align*}
        \log_\be\Biggl(\frac{1}{\be^{N(1-\al)}}\frac{1}{(\tau_\beta^N(x))^\al}\Biggr)&=-N(1-\al)-\al\log_\be (\tau_\beta^N(x)) \\
        &=-N\Bigl((1-\al)+\frac{\al}{N}\log_\be (\tau_\beta^N(x))\Bigr)\to-\infty
    \end{align*}
    as $N\to\infty$, we have that the rightmost side of the above inequality (\ref{above1}) is bounded above by a positive constant independent of $N$.   

Next, we show that there is a Lebesgue null set $\cN'\subset[0,1]$ such that $G_\be$ is $\al$-H\"older continuous from the right for any $\al\in(0,1)$ on $[0,1]\setminus \cN'$. 
By Lemma \ref{3} (2), we can take a Lebesgue null set $\cN'$ such that 
    \[\frac{1}{n}\log_{\be} (1-\tau_\beta^n(w))\to 0\]
    as $n\to\infty$ for any $w\in[0,1]\setminus \cN'$. Let $x\in(0,1)\setminus \cN'$ and take  $z\in(0,1)$ so that $x<z$. We write $N=N(x,z)$ for short. By Lemmas \ref{1} and \ref{2}(2), we obtain 
    \begin{equation}\label{above2}
    \begin{split}
        \frac{|G(z_{})-G(x_{})|}{(z-x)^\al}
        &\leq \frac{1}{(z-x)^\al M_\be}\Biggl|(z-x)+\sum_{n=N}^\infty\frac{ng_n(z)}{\be^n}\Bigl(M_\be-\frac{1}{n}\sum_{i=1}^ng_i(z)\Bigr)\\
        &-\sum_{n=N}^\infty\frac{ng_n(x)}{\be^n}\Bigl(M_\be-\frac{1}{n}\sum_{i=1}^ng_i(x)\Bigr)\Biggr|\\        &\leq \frac{(z-x)^{1-\al}}{M_\be}+\frac{\be^{\al N}}{M_\be(1-\tau_\beta^N(x))^\al}\Biggl|\sum_{n=N}^\infty\frac{ng_n(z)}{\be^n}\Bigl(M_\be-\frac{1}{n}\sum_{i=1}^ng_i(z)\Bigr) \\
        &-\sum_{n=N}^\infty\frac{ng_n(x)}{\be^n}\Bigl(M_\be-\frac{1}{n}\sum_{i=1}^ng_i(x)\Bigr)\Biggr| \\
        &\leq \frac{1}{M_\be}+\frac{\be^{\al N}}{M_\be(1-\tau_\be^N(x))^\al}\sum_{n=N}^\infty \frac{4n}{\be^n} \\
        &\leq \frac{1}{M_\be}+\frac{K'_\be}{M_\be\be^{(1-\al)N}}\frac{1}{(1-\tau_\beta^N(x))^\al}, 
        \end{split}
    \end{equation}
    where $K'_\be$ is a positive constant independent of $N$.
    Since 
    \begin{align*}
        &\log_\be\Bigl(\frac{1}{\be^{N(1-\al)}}\frac{1}{(1-\tau_\beta^N(x))^\al}\Bigr) \\
        &=-N(1-\al)-\al\log_\be (1-\tau_\beta^N(x)) \\
        &=-N\Bigl((1-\al)+\frac{\al}{N}\log_\be (1-\tau_\beta^N(x))\Bigr)\to-\infty
    \end{align*}
    as $N\to\infty$, we have that the rightmost side of the above inequality (\ref{above2}) is bounded above by a positive constant independent of $N$.

We have the conclusion by taking $\cF=(0,1)\setminus (\cN\cup\cN')$. 
\end{proof}

\section{Non-Lipschitz continuity}

This section is devoting to showing the following theorem:
\begin{theorem}\label{4-1}
    The function $G_\be(x)$ is not pointwise Lipschitz continuous at $\cL$-a.e. $x\in[0,1]$. 
\end{theorem}
To show this we apply the central limit theorem to the dynamics $(\tau_\beta,m_\be)$, which is established for piecewise $C^2$ expanding maps (see e.g., \cite{Ho-Ke}). Here we refer to the statement in \cite{Bo-Go}.

\begin{theorem}[Application of Theorem 8.5.1 in \cite{Bo-Go}]\label{4-2}
    Let $f$ be a function of bounded variation on $[0,1]$ and assume that $f$ is not of the form $f=c+\varphi\circ\tau_\beta-\varphi$, where $\varphi$ is of bounded variation and $c$ is a constant. Then 
    \[v^2:=\lim_{n\to\infty}\int_{0}^1 \Biggl(\frac{\sum_{i=0}^{n-1}f\circ\tau_\beta^i-n\cdot m_\be(f)}{\sqrt{n}}\Biggr)^2d m_\be>0\]
    and 
    \[\lim_{n\to\infty}m_\be\Biggl(\biggr\{\frac{\sum_{i=0}^{n-1}f\circ\tau_\beta^i-n\cdot m_\be(f)}{v\sqrt{n}}\leq r\biggr\}\Biggr)=\frac{1}{\sqrt{2\pi}}\int_{-\infty}^xe^{-\frac{t^2}{2}}dt\]
    for $x\in\R$. 
\end{theorem}

Let $S$ be the set of all simple numbers, i.e., the set of all $x\in[0,1]$ whose orbit by $\tau_\be$  eventually falls into $1/\be$. Note that $S$ is countable, which yields $\cL(S)=0$. For $x\in(0,1]\setminus S$, set $l(0)=0$ and $l(N)=\min\{n>l(N-1);\ g_n(x)=1\}$ for $N\geq1$. Since there are infinitely many $m\geq1$ such that $g_m(x)=1$, the sequence $(l(N))_{N=1}^\infty$ of positive integers is well-defined and strictly increasing. For $N\geq1$ let $x_N\in[0,1]$ be given by $x_N=\sum_{n=1}^{l(N)}g_n(x)/\be^n$. By definition $x_N<x_{N+1}$ for any $N\geq1$ and $x_N\to x$ as $N\to\infty$.

\begin{lemma}\label{4-3}
    For $\cL$-a.e. $x\in [0,1]\setminus S$, we have
    \[\limsup_{N\to\infty}\Biggl(l(N)-\frac{\sum_{i=1}^{l(N)}g_i(x)}{M_\be}\Biggr)=\infty.\]
\end{lemma}

\begin{proof}
    Set 
\[E_{\infty}=\Bigl\{x\in[0,1]\ ;\  \liminf_{N\to\infty}\sum_{i=1}^{N}(g_i(x)-M_\be)=-\infty\Bigr\}\]
and 
\[E_m=\Bigl\{x\in[0,1]\ ;\  \sum_{i=1}^{m}(g_i(x)-M_\be)\leq-\sqrt{mv_0^2}\Bigr\},\]
for $m\geq1$, where $v_0^2$ is the variation of $\chi_{[1/\be,1]}$:
\[v_0^2=\lim_{n\to\infty}\int_{0}^1 \Biggl(\frac{\sum_{i=0}^{n-1}\chi_{[1/\be,1]}\circ\tau_\beta^i-nM_\be}{\sqrt{n}}\Biggr)^2dm_\be.\]
We note that $\displaystyle{E_{\infty}\supset\bigcap_{n=1}^\infty\bigcup_{m=n}^\infty E_m}$.
Since the measurable dynamics $(\tau_\beta,m_\be)$ is ergodic and $\tau_\beta^{-1}E_{\infty}=E_{\infty}$, we have $m_\be(E_{-\infty})\in\{0,1\}$.  
In fact, the indicator function $\chi_{[1/\be,1]}$ is not of the form 
$\chi_{[1/\be,1]}=c+\psi\circ\tau_\beta-\psi$, where $\psi$ is of bounded variation. If such a function $\psi$ exists, we have that $c=M_\be$ and 
\begin{align*}
    \exp(2\pi i\psi\circ\tau_\beta)
    &=\exp(2\pi i \chi_{[1/\be,1]})\cdot \exp(-2\pi i M_\be)\cdot\exp(2\pi i\psi) \\
    &=\exp(-2\pi i M_\be)\cdot\exp(2\pi i\psi),
\end{align*}
which shows that $\exp(-2\pi i \cdot M_\be)$ is an eigenvalue of the Koopman operator defined by $Uf=f\circ\tau_\beta$ for $f\in L^2(m_\be)$. Since $(\tau_\beta, m_\be)$ is exact (see e.g., \cite{Bo, Sm}), in particular, it is weakly mixing. This yields that the Koopman operator has $1$ as its eigenvalue and there is no eigenvalue on the unit circle except $1$ (see \cite{Wa}), which yields $M_\be\in\{0,1\}$. By the fact that $m_\be$ is equivalent to the Lebesgue measure and has full support on $[0,1]$, however, we know that $M_\be\in(0,1)$, which yields a contradiction. 

By Theorem \ref{4-2},
we obtain that there is a positive integer $n_0$ such that 
\[m_\be(E_{n})>\frac{1}{2}\int_{-\infty}^{-1}\ e^{-\frac{t^2}{2}}dt>0\]
for $n\geq n_0$. Hence 
\[m_\be(E_{\infty})\geq \liminf_{n\to\infty}m_\be\Bigl(\bigcup_{m=n}^\infty E_m\Bigr)\geq \frac{1}{2}\int_{-\infty}^{-1}e^{-\frac{t^2}{2}}dt>0,\]
which yields $m_\be(E_{\infty})=1$. Together with that $m_\be$ is equivalent to $\cL$, we obtain $\cL(E_\infty)=1$. 

Let $x\in E_{\infty}\setminus S$. Then for any positive integer $L$ there is a positive integer $n_0$ such that 
\[\sum_{i=1}^{n_0}(g_i(x)-M_\be)\leq -L.\]
Let $N_L$ be the minimal positive integer such that $l(N_L)>n_0$. Since $g_i(x)=0$ for $i\notin(l(N))_{N=1}^\infty$ we have

\begin{align*}
&\sum_{i=1}^{l(N_L)}(g_i(x)-M_\be) \\
&=\sum_{i=1}^{n_0}(g_i(x)-M_\be)+\sum_{i=n_0+1}^{l(N_L)}(g_i(x)-M_\be) \\
&\leq -L+1-(l(N_L)-n_0)M_\be\leq -L+1,
\end{align*}
which yields 
\begin{align*}
\limsup_{N\to\infty}\Biggl(l(N)-\frac{\sum_{i=1}^{l(N)}g_i(x)}{M_\be}\Biggr)&=-\frac{1}{M_\be}\liminf_{N\to\infty}\Biggl(\sum_{i=1}^{l(N)}(g_i(x)-M_\be)\Biggr) \\
&= \infty. 
\end{align*}
This gives the conclusion.  
\end{proof}

\begin{proof}[Proof of Theorem \ref{4-1}]
By Lemma \ref{4-3}, we have 
\begin{equation}\label{supL}
\limsup_{N\to\infty}\Biggl(l(N)-\frac{\sum_{i=1}^{l(N)}g_i(x)}{M_\be}\Biggr)=\infty
\end{equation}
for $x\in[0,1]\setminus (S\cup \cN)$, where $\cN$ is a Lebesgue null set and $(l(N))_{N=0}^\infty$ is a sequence of non-negative integers given by $l(0)=0$ and $l(N)=\min\{n>l(N-1);\ g_n(x)=1\}$ for $N\geq1$.
Assume that $G(x)$ is Lipschitz continuous at $x\in[0,1]\setminus(S\cup \cN)$. Then there is a positive constant $K$ such that 
\[\frac{|G_\be(x_N)-G_\be(x)|}{|x_N-x|}\leq K\]
for any $N\geq1$, where $x_N=\sum_{n=1}^{l(N)}g_n(x)/\be^n$. 

On the one hand, we have 
\begin{align*}
    \frac{|G_\be(x_{N+1})-G_\be(x_{N})|}{|x_{N+1}-x_N|}
    &\leq\frac{|G_\be(x_{N+1})-G_\be(x_{})|}{|x_{N+1}-x_N|}+\frac{|G_\be(x_{N})-G_\be(x_{})|}{|x_{N+1}-x_N|} \\ \notag
    &\leq \frac{|x_{N+1}-x|}{|x_{N+1}-x_N|}K+\frac{|x_{N}-x|}{|x_{N+1}-x_N|}K \\ \notag
    &\leq \frac{1}{1-1/\be}\frac{\be^{l(N+1)}}{\be^{l(N+2)}}K+\frac{1}{1-1/\be}\frac{\be^{l(N+1)}}{\be^{l(N+1)}}K \\
    &\leq \frac{2}{1-1/\be}K, 
\end{align*}
which shows $|G_\be(x_{N+1})-G_\be(x_{N})|/|x_{N+1}-x_N|$ is bounded for $N\geq1$. 
On the other hand, since
\begin{align*}
    \frac{|G_\be(x_{N+1})-G_\be(x_{N})|}{|x_{N+1}-x_N|}
    &=\be^{l(N+1)}\cdot\frac{1}{\be^{l(N+1)}}\Biggl(l(N+1)-\frac{\sum_{i=1}^{l(N+1)-1}g_i(x)}{M_\be}\Biggr) \\
    &=l(N+1)-\frac{\sum_{i=1}^{l(N+1)}g_i(x)}{M_\be}+\frac{1}{M_\be}, 
\end{align*}
we obtain 
\[\limsup_{N\to\infty}\frac{|G(x_{N+1})-G(x_{N})|}{|x_{N+1}-x_N|}=\infty\]
by (\ref{supL}), which yields a contradiction. 
Since the set $S\cup\cN$ is a Lebesgue null set on $[0,1]$, we have the desired result.
\end{proof}

Acknowledgments: This work was supported by JSPS KAKENHI Grant Number 24K16932.

\bibliographystyle{amsplain}

\end{document}